\documentclass[11pt,a4paper,oneside]{amsart}
\usepackage{amsmath,amssymb}
\usepackage{comment}
\DeclareMathOperator{\esssup}{ess\,sup}
\DeclareMathOperator{\diameter}{diam}

\newcommand{\dist}{{\operatorname{dist}}}

\DeclareMathOperator{\diam}{diam}

\DeclareMathOperator{\interior}{int}
\newtheorem{theorem}{Theorem}[section]
\newtheorem{lemma}[theorem]{Lemma}

\newtheorem{definition}[theorem]{Definition}
\newtheorem{remark}[theorem]{Remark}

\numberwithin{equation}{section}
\def\Xint#1{\mathchoice 
 {\XXint\displaystyle\textstyle{#1}}%
{\XXint\textstyle\scriptstyle{#1}}%
{\XXint\scriptstyle\scriptscriptstyle{#1}}%
 {\XXint\scriptscriptstyle\scriptscriptstyle{#1}}%
 \!\int}
\def\XXint#1#2#3{{\setbox0=\hbox{$#1{#2#3}{\int}$}
 \vcenter{\hbox{$#2#3$}}\kern-.5\wd0}}

 \def\dashint{\Xint-}

\DeclareMathOperator{\essinf}{ess inf}

\title{Capacity and Hausdorff measure in Musielak-Orlicz-Sobolev spaces}
\author{ Ankur Pandey, Nijjwal Karak,  Debarati Mondal}
\thanks{Ankur Pandey gratefully acknowledges Council of Scientific and Industrial Research (CSIR) for awarding Junior Research Fellowship (e-certificate no. 22D/23J05002) and Debarati Mondal also acknowledges Council of Scientific and Industrial Research (CSIR) for awarding Senior Research Fellowship (file no. 09/1026(18426)/2024-EMR-I)}
\begin{document}

\begin{abstract}
   In this paper, we show that sets with zero Sobolev $p(\cdot)$-capacity have generalized Hausdorff $h(\cdot)$-measure zero, for some gauge function $h(\cdot).$ We also prove that sets with zero Musielak-Orlicz-Sobolev $\Phi(\cdot,\cdot)$-capacity, for a particular class of functions $\Phi(\cdot,\cdot),$ have generalized Hausdorff $h(\cdot)$-measure zero, for a suitable gauge function $h(\cdot).$ These results can be seen as improvements of the earlier results in \cite{BHH18} and \cite{HHKV03}.
\end{abstract}
\maketitle
\section{Introduction}
Sobolev $p(\cdot)$-capacity in variable exponent Sobolev spaces has been introduced by Harjulehto, H\"ast\"o, Koskenoja and Varonen in \cite{HHKV03}, and many basic properties have also been established there. Let us recall the definition here.
\begin{definition}
    Let $E\subset\mathbb{R}^n,$ $n\geq 2.$ The Sobolev $p(\cdot)$-capacity of $E$ is defined by
    \begin{equation*}
        C_{p(\cdot)}(E):= \inf_{u\in S_{p(\cdot)}(E)}\int_{\mathbb{R}^n} |u|^{p(x)}+ |\triangledown u|^{p(x)} dx,
    \end{equation*}
    where $ S_{p(\cdot)}(E)=\{u \in W^{1,p(\cdot)}(\mathbb{R}^n):u\geq 1  $ in an open set containing $E$ and $u\geq 0\}.  $
\end{definition}
\noindent It was proved in the same paper that the Sobolev $p(\cdot)$-capacity is an outer capacity, and it also satisfies the Choquet proeprty when the exponent $p(\cdot)$ is bounded away from $1$ and $\infty.$ It was further studied by Harjulehto, H\"ast\"o \cite{HH04}, Harjulehto, H\"ast\"o, Koskenoja, Varonen \cite{HHKV05}, Harjulehto, H\"ast\"o, Latvala \cite{HHL04}, and Mashiyev \cite{Mas08}. Relative Sobolev capacity in variable exponent Sobolev space has also been studied and compared with the Sobolev $p(\cdot)$-capacity in \cite{HHK07}.\\

In this paper, we study the relation between the Sobolev $p(\cdot)$-capacity and the generalized Hausdorff measure. Recall that, for the classical Sobolev space $W^{1,p}(\mathbb{R}^n),$ sets of $p$-capacity zero have generalized Hausdorff $h$-meaure zero provided that
\begin{equation*}
    \int_0^1(t^{p-n}h(t))^{\frac{1}{p-1}}\,\frac{dt}{t}<\infty,
\end{equation*}
for $1<p\leq n,$ see \cite[Theorem 5.1.13]{AH96} or \cite[Theorem 7.1]{KM72}. See also \cite[Chapter 4.7.2]{EG92} for more related results. For the variable exponent Sobolev space, the authors in \cite{HHKV03} proved that sets with zero $p(\cdot)$-capacity have $s$-Hausdorff measure zero for all $s>n-p^-,$ where $p^-:=\essinf_{x\in \mathbb{R}^n}p(x).$ This result was improved in the book \cite{DHHR11} using the variable Hausdorff measure.
\begin{theorem}[Theorem 10.4.7, \cite{DHHR11}]
Suppose that $p:\mathbb{R}^n\rightarrow [1,n]$ satisfies both the log-H\"older continuity and the log-H\"older decay condition and let $s:\mathbb{R}^n\rightarrow (0,n]$ be continuous with $s(\cdot)>n-p(\cdot).$ If $C_{p(\cdot)}(E)=0,$ then $H^{s(\cdot)}(E)=0.$    
\end{theorem}
\noindent Here we have further improved this result using generalized Hausdorff $h(\cdot)$-measure. See Section 2 for the definitions.
\begin{theorem}\label{main theorem 2}
   Suppose that $p:\mathbb{R}^n\rightarrow [1,n]$ satisfies both the log-H\"older continuity and the log-H\"older decay condition and let $h:\mathbb{R}^n\times (0, \infty)  \rightarrow (0, \infty) $ be given by $h(\cdot, t)=t^{n-p(\cdot)}\log^{-p^+-\delta}\left(\frac{1}{t}\right)$ for any $\delta>0$. If $C_{p(\cdot)}(E)=0$ for any $E \subset \mathbb{R}^n$, then $\mathcal H^{h(\cdot)}(E)=0$.  
\end{theorem}
\noindent We also get a lower bound of the $p(\cdot)$-capacity of a compact set $E$ inside a ball in terms of generalized Hausdorff $h(\cdot)$-measure, using the techniques of \cite{KK15revista}.
\begin{theorem}\label{main theorem 1}
 Let $x_0\in \mathbb{R}^n, 0<R<\infty$ and $E\subset B(x_0,R)$ be any compact set. Suppose that $p:\mathbb{R}^n\rightarrow [1,n]$ satisfies both the log-H\"older continuity and the log-H\"older decay condition. Then there exists a constant $c>0$ depending only on $p,$ $n$ and $R$ such that $\mathcal H^{h(\cdot)}(E) \leq c C_{p(\cdot)}(E),$ where $h(\cdot,t)=t^{n-p(\cdot)}(\log(\frac{1}{t}))^{{-p^+}-\delta}$ for any $\delta >0.$   
\end{theorem}
\vspace{0.3cm}
We also consider here the relation between the Musielak-Orlicz-Sobolev capacity, which was studied earlier in \cite{BHH18}, and the Hausdorff measure. Let us recall the definition of the capacity here.
\begin{definition}
    The Musielak-Orlicz-Sobolev  $\Phi(\cdot, \cdot)$-capacity of $E\subset \mathbb{R}^n$ is defined by 
    \begin{equation}
        C_{\Phi(\cdot, \cdot)}(E):= \inf_{u\in S_{\Phi(\cdot, \cdot)}(E)}\int_{\mathbb{R}^n} \Phi(x, u)+ \Phi(x, \triangledown u)  dx
    \end{equation}
    where $S_{\Phi(\cdot, \cdot)}(E):= \{u\in W^{1,\Phi(\cdot, \cdot)}(\mathbb{R}^n): u\geq 1  $ in an open set containing $E$ and $u\geq 0\}.$
\end{definition}
\noindent Here we consider the class of functions $\Phi(x,t)=t^{p(x)}\log^{q(x)}(e+t)$ as generalized weak $\Phi$-functions, where $1\leq p(x)\leq n$ and $-\infty<q(x)<\infty$ for all $x\in\mathbb{R}^n.$ Again we refer to Section 2 for the definition of weak $\Phi$-functions. We have the following existing result for a general class of functions.
\begin{theorem}\cite[Corollary 14]{BHH18}
    Let $\Phi\in\Phi_{w}(\mathbb{R}^n)$ satisfy $(A0),$ $(aInc)_p$ with $p>1$ and $(aDec).$ If $p\leq n$ and $E\subset\mathbb{R}^n$ with $C_{\Phi(\cdot, \cdot)}(E)=0,$ then $H^s(E)=0$ for all $s>n-p.$
\end{theorem}
\noindent Note that the function $\Phi(x,t)=t^{p(x)}\log^{q(x)}(e+t)$ satisfies the conditions $(A0),$ $(aDec)$ (see \cite{PK24} for the proofs) and also $(aInc)_{p^-}$ when $q(x)\geq 0.$ Therefore, when $q(x)\geq 0,$ we have from the previous theorem that $H^s(E)=0$ for all $s>n-p^-$ and all sets $E\subset\mathbb{R}^n $ with $C_{\Phi(\cdot, \cdot)}(E)=0.$ We improve this result here with some additional assumptions on $p.$
\begin{theorem}\label{main theorem 3}
   Suppose that $p:\mathbb{R}^n\rightarrow [1,n]$ satisfies both the log-H\"older continuity and the log-H\"older decay condition and $q:\mathbb{R}^n\rightarrow (0,\infty)$ is any function. Let $h:\mathbb{R}^n\times (0, \infty)  \rightarrow (0, \infty) $ be given by $h(\cdot, t)=t^{n-p(\cdot)}\log^{-p^+}\left(\frac{1}{t}\right)$ and let $\Phi(x, t) = t^{p(x)}\log^{q(x)}(e+t)$. If $C_{\Phi(\cdot, \cdot)}(E)=0$ for any $E \subset \mathbb{R}^n$, then $\mathcal H^{h(\cdot)}(E)=0$.
\end{theorem}
\noindent We also get a similar result, when $q(x)<0$ for all $x\in\mathbb{R}^n.$
\begin{theorem}\label{main thorem 4}
   Suppose that $p:\mathbb{R}^n\rightarrow [1,n]$ satisfies both the log-H\"older continuity and the log-H\"older decay condition and $q:\mathbb{R}^n\rightarrow (-\infty, 0)$ is any function. Let $h:\mathbb{R}^n\times (0, \infty)  \rightarrow (0, \infty) $ be given by $h(\cdot, t)=t^{n-p(\cdot)}\log^{-p^++\frac{q^+ p^+}{p^-}-\delta}\left(\frac{1}{t}\right),$ for any $\delta >0$, and let $\Phi(x, t) = t^{p(x)}\log^{q(x)}(e+t)$. If $C_{\Phi(\cdot, \cdot)}(E)=0$ for any $E \subset \mathbb{R}^n$, then $\mathcal H^{h(\cdot)}(E)=0$.
\end{theorem}
Throughout the paper, we denote by $c$ a positive constant which is independent of the main parameters, but which may vary from line to line. The symbol $A\lesssim B$ means that $A\leq cB$ for some constant $c>0.$ If $A\lesssim B$ and $B\lesssim A,$ then we write $A\approx B.$
\section{Notations and Preliminaries }
Let us first recall the definition of variable exponents and their related standard notations.
\begin{definition}
    Let $\Omega\subset \mathbb{R}^n$ and $(\Omega, \Sigma, \mu)$ be a $\sigma$-finite, complete measure space. We define $\mathcal{P}(\Omega, \mu) $ to be the set of all $\mu$-measurable functions $p:\Omega\rightarrow [1,\infty).$ Functions $p\in \mathcal{P}(\Omega, \mu)$ are called variable exponents on $\Omega$.\\
    In the special case that $\mu$ is the $n$-dimensional Lebesgue measure and $\Omega$ is an open subset of $\mathbb{R}^n$, we abbreviate $\mathcal{P}(\Omega):= \mathcal{P}(\Omega, \mu). $
\end{definition}
\begin{definition}
Let\: $\Omega\subset \mathbb{R}^n.$ A variable exponent is a bounded measurable function $p:\Omega\rightarrow [1,\infty),$ usually denoted by $p(\cdot).$\\
For such $p(\cdot),$ where $A$ is a measurable subset of $\Omega$, $$p_A^+:=\esssup\{p(x):x\in A\} \;\;\text{and}\;\; p_A^-:=\essinf\{p(x):x\in A\}.$$
Also we denote $$p^+:=\esssup_{x\in \mathbb{R}^n}p(x)\;\;\text{and}\;\;p^-:=\essinf_{x\in \mathbb{R}^n}p(x).$$
\end{definition}
\noindent Throughout the paper, we assume that $p$ is bounded in $\mathbb{R}^n$, that is $1\leq p^-\leq p^+<\infty.$\\

For the study of variable exponent Lebesgue spaces and Sobolev spaces, we refer the reader to the books \cite{CF13, DHHR11, PKJF13}. We will need the following H\"older inequality on variable exponent Lebesgue spaces; the proof can be found in any of the above references.
\begin{lemma}
 If $r(\cdot)\,,\, p(\cdot)\,,\,q(\cdot) $ are variable exponents satisfying $1/r(x)=1/p(x)+1/q(x)$ almost everywhere on $\Omega$ then \begin{equation}   \label{holder inequality}     
    \Vert uv\Vert_{L^{r(\cdot)}(\Omega)}\leq 2 \Vert u\Vert_{L^{p(\cdot)}(\Omega)}\,\Vert v\Vert_{L^{q(\cdot)}(\Omega)}.\end{equation}
\end{lemma}
\noindent We recall the first order variable exponent Sobolev space here.
\begin{definition}
    The function $u\in L^{p(\cdot)}(\Omega) $ belongs to space  ${ W^{1,p(\cdot)}(\Omega)}$, if its weak partial derivatives $\delta_1{u},.....,\delta_n{u}$ exist and belong to $L^{p(\cdot)}(\Omega)$. We define a semimodular on ${ W^{1,p(\cdot)}(\Omega)}$ by 
    \begin{center}
     $\rho_{W^{1,p(\cdot)}(\Omega)}(u) := \rho_{p(\cdot)}(u) + \rho_{p(\cdot)}(\triangledown u)= \int_{\Omega} |u|^{p(x)}+ |\triangledown u|^{p(x)} dx$.
 \end{center}
 It induces a (quasi-) norm
 \begin{center}
$ ||u||_{W^{1,p(\cdot)} (\Omega)} := \inf  \{ \lambda >0 : \rho_ {W^{1,p(\cdot)} (\Omega)}\left(\frac{u}{\lambda}\right)\leq 1\}$.
 \end{center}
  The space $W^{1,p(\cdot)} (\Omega)$ is called Sobolev space and its elements are called Sobolev functions.
\end{definition}
Now, let us recall some terminology to define Musielak-Orlicz spaces and Musielak-Orlicz-Sobolev spaces; more details about these spaces can be found in the books \cite{HH19, ML19}.
\begin{definition}
 A function $f:(0,\infty) \rightarrow \mathbb{R}$ is almost increasing if there exists a constant $a\geq 1$ such that $f(s) \leq af(t)$ for all $0< s< t .$ Similarly, a function $f:(0,\infty) \rightarrow \mathbb{R}$ is almost decreasing if there exists a constant $b\geq 1$ such that $f(s) \geq bf(t)$ for all $0< s< t .$ 
 \end{definition}
 \begin{definition}
Let $f:(0,\infty)\rightarrow\mathbb{R}$ and $p,q>0$.We say that f satisfies
\begin{enumerate}
\item[(i)]  $(Inc)_{p}$      if $f(t)/t^p$ is increasing; 
\item[(ii)] $(aInc)_{p}$     if $f(t)/t^p$ is almost increasing;
\item[(iii)]$(Dec)_{q}$      if $f(t)/t^q$ is decreasing;
\item[(iv)] $(aDec)_{q}$     if $f(t)/t^q$ is almost decreasing.
\end{enumerate}
\end{definition}
We say that $f$ satisfies $(aInc)$, $(Inc)$, $(aDec)$ or $(Dec)$ if there exists $p>1$ or $q<\infty$ such that $f$ satisfies $(aInc)_{p}$,$(Inc)_{p}$, $(aDec)_{q}$ or $(Dec)_{q}$, respectively.  
\begin{definition}
   Let $(\Omega,\Sigma,\mu)$ be a $\sigma$-finite, complete measure space. A function $\Phi: \Omega\times [0,\infty)\rightarrow [0,\infty]$ is said to be a (generalized) $\Phi$-prefunction on $(\Omega,\Sigma,\mu)$ if $x\rightarrow \Phi(x,|f(x)|)$ is measurable for every $f\in L^0(\Omega,\mu) $ and $\Phi(x,.)$ is a $\Phi$-prefunction for $\mu$-almost every $x\in \Omega $. We say that the $\Phi$-prefunction $\Phi$ is  a (generalized weak) $\Phi$-function if it satisfies $(aInc)_1$. 
The set of generalized weak $\Phi$-functions is denoted by $\Phi_w(\Omega,\mu)$.
\end{definition}
\begin{definition}
     We say that $\Phi\in\Phi_w(\Omega,\mu)$ satisfies $(A0)$ if there exists a constant $\beta\in(0,1]$ such that
   \begin{center} 
     $\beta\leq\Phi^{-1}(x,1)\leq{1/\beta}$  \quad for $\mu $-almost every $x\in \Omega$.
     \end{center}
\end{definition}
\begin{definition}
Let $\Phi\in\Phi_w(\Omega,\mu)$ and let $\rho_\Phi $ be given by
\begin{center}
  $\rho_\Phi(f)  := \int_\Omega\Phi(x,|f(x)|)d\mu(x)$
\end{center}	
for all $f\in L^0(\Omega,\mu)$. The function $\rho_\Phi $ is called a modular. The set 
\begin{equation*}
 L^{\Phi(\cdot,\cdot)}(\Omega,\mu) := \{f\in L^0(\Omega,\mu) :\rho_\Phi(\lambda f)< \infty \,\, for \,\, some \,\, \lambda > 0  \}  
 \end{equation*}
 is called a generalized Orlicz space or Musielak-Orlicz (M-O) space.
 \end{definition}
 \begin{definition}
 Let $\Phi\in\Phi_w(\Omega,\mu)$. The function $u\in L^{\Phi(\cdot,\cdot)}\cap L^1_{loc}(\Omega)$ belongs to the Musielak-Orlicz-Sobolev space $W^{1,\Phi(\cdot,\cdot)}(\Omega)$ if its weak partial derivatives $\delta_1{u},.....,\delta_n{u}$ exist and belong to $L^{\Phi(\cdot,\cdot)}(\Omega)$. We define a semimodular on $W^{1,\Phi(\cdot,\cdot)}(\Omega)$ by 
 \begin{center}
     $\rho_{W^{1,\Phi(\cdot,\cdot)}(\Omega)}(u) := \rho_\Phi(u) + \rho_\Phi(\triangledown u)$.
 \end{center}
 It induces a (quasi-) norm
 \begin{center}
$ ||u||_{W^{1,\Phi(\cdot,\cdot)} (\Omega)} := \inf  \{ \lambda >0 : \rho_ {W^{1,\Phi(\cdot,\cdot)} (\Omega)}\left(\frac{u}{\lambda}\right)\leq 1\}$.
 \end{center}
 \end{definition}
 Let us also recall the following three regularity conditions on the exponents.
\begin{definition}
    We say that the exponent $p(\cdot)$ is log-H\"older continuous if there exists a constant $c_1>0$ 
    \begin{center}
    $|p(x)-p(y)|\leq \frac{c_1}{\log(e+1/|x-y|)}$ whenever $x\in\mathbb R^n$ and $y\in\mathbb R^n.$
    \end{center}
\end{definition}
\begin{definition}
     We say that the exponent $p(\cdot)$ satisfies the  log-H\"older decay condition if there exist $p_{\infty}\in \mathbb{R}$ and a constant $c_2>0$ such that 
    \begin{center}
    $|p(x)-p_{\infty}|\leq \frac{c_2}{\log(e+|x|)}$ for all $x\in\mathbb R^n$.
    \end{center}
\end{definition}
\begin{definition}
     We say that the exponent $p(\cdot)$ is log-log-H\"older continuous if there exists a constant $c_3>0$ 
     \begin{center}
         $|p(x)-p(y)|\leq \frac{c_3}{\log(e+\log(e+1/|x-y|))}$ whenever $x\in\mathbb R^n$ and $y\in\mathbb R^n $.
     \end{center}
\end{definition}
\noindent The following two lemmas are easy consequences of the regularity conditions of the exponents.
\begin{lemma}\label{from logholder continuity}
    Let $p:\mathbb{R}^n\rightarrow \mathbb{R}$ be bounded and log-H\"older continuous and $B \subset  \mathbb{R}^n$ be any ball of radius r. If $p(x)\geq p(y)$  for any two points $x, y\in B$, then there is a constant $c_0 >0$ such that $r^{p(y) -p(x)}\leq c_0.$
\end{lemma}
\begin{proof}
    By log-H\"older continuity of $p(\cdot)$ we get 
    \begin{center}
        $\big(p(x)-p(y)\big) \leq \frac{c_1}{\log(e+1/|x-y|)}\leq \frac{c_1}{\log(e+1/2r)}$
    \end{center}
    which implies $\frac{r^{p(y) -p(x)}}{2^{p(x) -p(y)}}\leq (e+1/2r)^{p(x) -p(y)}\leq C_1, $ where $C_1= e^{c_1}$. 
    Therefore, we get $r^{p(y) -p(x)}\leq c_0$, where $c_0 = C_1 2^{(p^+ - p^-)}$.
    
\end{proof}
\begin{lemma}\label{from loglogholder continuity}
    Let $p:\mathbb{R}^n\rightarrow \mathbb{R}$ be bounded and log-log-H\"older continuous and $B \subset  \mathbb{R}^n$ be any ball of radius $r<1$. If $p(x)\geq p(y)$  for any two points $x, y\in B$, then there is a constant $c_4 >0$ such that $(\log(1/r))^{p(x) -p(y)}\leq c_4.$
\end{lemma}
\begin{proof}
    By log-log-H\"older continuity of $p(\cdot)$ we get 
    \begin{center}
        $\big(p(x)-p(y)\big) \leq \frac{c_3}{\log(e+\log(e+1/|x-y|))}\leq \frac{c_3}{\log(e+\log(e+1/2r))}$
    \end{center}
    which implies $(\log(1/2r))^{p(x) -p(y)}\leq (e+\log(e+1/2r))^{p(x) -p(y)}\leq C_3, $ where $C_3= e^{c_3}$. 
    Therefore, we get $(\log(1/r))^{p(x) -p(y)} \leq c_4$.
    
\end{proof}
\noindent We borrow the following two lemmas from \cite{DHHR11}.
\begin{lemma}\cite[Lemma 4.1.6]{DHHR11}\label{logholder continuous}
     Let $p:\mathbb{R}^n\rightarrow \mathbb{R}$ be bounded and satisfy the log-H\"older continuity condition. If $B \subset  \mathbb{R}^n$ is any ball of radius r then for any $x\in B $ we have $|B|^{\frac{1}{p_{B}^+}-\frac{1}{p(x)}} \leq c$.
\end{lemma}
\begin{lemma}\cite[Theorem 4.5.7]{DHHR11}\label{norm of characteristic}
     Let $p:\mathbb{R}^n\rightarrow \mathbb{R}$ be bounded and log-H\"older continuous then $|| 1_B||_{L^{p(\cdot)}} \approx |B|^{\frac{1}{p_{B}}}$ for all balls $B \subset \mathbb{R}^n$, where $p_B = \dashint_B p(y) dy$.
\end{lemma}
Next, we define the generalized Hausdorff $h(\cdot)$-measure. Note that the variable dimension Hausdorff measure was first introduced in \cite{HHL06}.
\begin{definition}
    Let $h:\mathbb{R}^n\times(0,\infty)\rightarrow(0, \infty) $ be a non-decreasing function such that $\lim_{t\rightarrow0}h(\cdot,t)=0.$ Then the Hausdorff h($\cdot$)-measure of $E\subset \mathbb{R}^n$ is defined by
    \begin{equation}
    H^{h(\cdot)}(E):= \lim_{\eta\rightarrow0} H^{h(\cdot)}_{\eta}(E),
    \end{equation}
 where $ H^{h(\cdot)}_{\eta}(E)=\inf \{ h(x_i,\diam(B_i)): E\subset \cup_{i}B_i, 0< \diam(B_i)\leq \eta\},$ where $x_i$ is the center of the ball $B_i.$ If $s$ is a continuous function on $\mathbb{R}^n$ and $h(\cdot, t) = t^{s(\cdot)}$, then we denote $ H^{s(\cdot)}(E):=  H^{h(\cdot)}(E).$
\end{definition}
For the proof of the 5B-covering lemma (or Vitali's covering lemma), see for example, \cite{EG92}.
\begin{lemma}(5B-covering lemma) \label{b covering lemma}
    Let $\mathcal{B}$ be a collection of open balls in $\mathbb{R}^n$ with uniformly bounded radii. Then there exists a disjoint sub-collection $\mathcal{B'}$ of $\mathcal{B}$ such that 
    $$ \bigcup_{B\in \mathcal{B}} B \subset \bigcup_{B\in \mathcal{B'}} 5B.$$
\end{lemma}
Finally, we state the following lemma, which will be used for the proof of Theorem \ref{main theorem 1}.
\begin{lemma}\cite[Lemma 3.7]{HHKV03}\label{admissible continuous}
  Let $p(\cdot)$ be a bounded exponent and assume that $C^\infty(\mathbb{R}^n)\cap W^{1,p(\cdot)}(\mathbb{R}^n)$ is dense in  $W^{1,p(\cdot)}(\mathbb{R}^n).$ If $K$ is compact, then 
 \begin{equation*}
        C_{p(\cdot)}(K):= \inf_{u\in S'_{p(\cdot)}(K)}\int_{\mathbb{R}^n} |u|^{p(x)}+ |\triangledown u|^{p(x)} dx,
    \end{equation*}
  where  $S'_{p(\cdot)}(K)= S_{p(\cdot)}(K)\cap C^\infty(\mathbb{R}^n).$
\end{lemma}

        \section{Main Results }
We start with a lemma, the proof of which uses the techniques of \cite[Chapter 2.4.3]{EG92}.        
\begin{lemma} \label{main lemma}
    Let $p\in \mathcal{P}(\mathbb{R}^{n})$ and let $h:\mathbb{R}^n\times(0, \infty)  \rightarrow (0, \infty) $ be any bounded function such that $\lim_{t\rightarrow0} h(\cdot, t)=0$. Then 
    \begin{equation*}
         H^{h(\cdot)}\big(\{{x\in \mathbb{R}^{n}: \limsup_{r\rightarrow0}\frac{1}{h(x, r)}}\int_{B(x, r)} |f|^{p(y)} dy = \infty\}\big) =0 
    \end{equation*}
    for every $f\in L^{p(\cdot)}(\mathbb{R}^{n})$.
\end{lemma}
\begin{proof}
    Let $\lambda >0$ and consider 
    \begin{equation*}
    E_{\lambda}:= \{ \limsup_{r\rightarrow0}\frac{1}{h(x, r)}\int_{B(x, r)} |f|^{p(y)} dy > \lambda\}.
    \end{equation*}
    Fix $\eta>0.$ Then for every $x\in E_{\lambda} $ we get a number $r_{x} \leq \eta$ such that 
        \begin{equation}\label{from the definition of Elambda}
        \int_{B(x, r_{x})} |f|^{p(y)} dy > \lambda h(x, r_x).
        \end{equation}
    By the $5B$-covering lemma (Lemma \ref{b covering lemma}), there is a countable subfamily of pairwise disjoint balls $B(x_i, 
 r_{x_i})$ such that $E_{\lambda} \subset \cup_{i=1}^{\infty} B(x_i, 5r_{x_i}) $. From the definition of Hausdorff measure, after using \eqref{from the definition of Elambda} for the balls $B(x_i, r_{x_i})$ and the fact that $E\subset E_{\lambda} $ for all $\lambda>0$, we get 
    \begin{eqnarray*}
          H^{h(\cdot)}_{\eta}(E)&\leq & c \sum_{i} h(x_i, r_{x_i})\leq c \sum_{i} \lambda^{-1} \int_{B(x_i,  r_{x_i})} |f|^{p(y)} dy\\
         &\leq& c \lambda^{-1} \int_{\mathbb{R}^{n}} |f|^{p(y)} dy.\\
    \end{eqnarray*}
    Then, by taking $\eta\rightarrow0$ we obtain $$ H^{h(\cdot)}(E) \leq c \lambda^{-1} \int_{\mathbb{R}^{n}} |f|^{p(y)} dy,$$ and finally the claim follows by choosing $\lambda$ arbitrarily large.
\end{proof}

\textbf{Proof of Theorem \ref{main theorem 2}}
    Since $C_{p(\cdot)}(E)=0,$ for every $i\in\mathbb{N},$ we have a non-negative function $u_i \in W^{1,p(\cdot)}(\mathbb{R}^n)$ such that $E\subset \interior\{u_i\geq 1\}$ and 
     $$\int_{\mathbb{R}^n} \big(|u|^{p(y)}+ |\triangledown u|^{p(y)}\big) dy\leq 2^{-ip^+}$$
    which implies that
    $$\int_{\mathbb{R}^n} \Big(\frac{|u|}{2^{-i}}\Big)^{p(y)}+ \Big(\frac{|\triangledown u|}{2^{-i}}\Big)^{p(y)} dy\leq 1$$
    and hence from the definition of norm, $||u_i||_{W^{1,p(\cdot)}(\mathbb{R}^n)}\leq 2^{-i}.$ Define $v_i:= u_1+....+u_i.$ Then $(v_i)$ is a Cauchy sequence in $W^{1,p(\cdot)}(\mathbb{R}^n),$ and since $W^{1,p(\cdot)}(\mathbb{R}^n)$ is a Banach space, there exists $u\in W^{1,p(\cdot)}(\mathbb{R}^n)$ such that $v_i\rightarrow u$ and hence
    \begin{center}
        $||u||_{1,p(\cdot)}\leq \lim_{i\rightarrow \infty} ||v_i||_{1,p(\cdot)}\leq \sum_{i=1}^{\infty}||u_i||_{1,p(\cdot)}\leq\sum_{i=1}^{\infty}2^{-i}\leq 1. $
    \end{center}
    For every natural number $k$ and $x\in E$ we find $r$ such that $u_{i}(y) \geq 1$ for every $i= 1,....,k$ and almost every $y\in B(x, r),$ which means that $E\subset \interior\{u\geq m\}$ for all $m\geq 1$. Then for $r$ small enough we get $B(x, r) \subset \interior\{u\geq m\}$, and hence $ \dashint_{B(x, r)} u dy \geq m$. Therefore
    \begin{equation}\label{main equation}
        \limsup_{r\rightarrow0}\dashint_{B(x, r)} u dy = \infty.
    \end{equation}
    Suppose that $x$ is such that 
    \begin{equation}
        \limsup_{r\rightarrow0} r^{\frac{-(n-p(x))}{p(x)}}\log^{\frac{p^++\delta}{p^+}}\left(\frac{1}{r}\right)\Vert\triangledown u\Vert_{L^{p(\cdot)}(B(x,r))}=: C <\infty.
    \end{equation}
     We choose $R\in (0, 1) $ so that, for every $0<r\leq R$, 
    \begin{equation*}
        \Vert\triangledown u\Vert_{L^{p(\cdot)}(B(x,r))}< C r^{\frac{n-p(x)}{p(x)}}\log^{\frac{-p^+-\delta}{p^+}}\left(\frac{1}{r}\right).
    \end{equation*}
  Denote $B_i:= B(x, 2^{-i}R)$. Then using the classical Poincar\'e inequality and the H\"older inequality \eqref{holder inequality}, we have  
    \begin{eqnarray*}
        |{ u}_{B_{i+1}} - { u}_{B_{i}}|
        &\leq& \dashint_{B_{i+1}}|u-{ u}_{B_{i}}| dy\\ &\leq& c2^{(i+1)n} \int_{B_i}|u-{ u}_{B_{i}}| dy\\
        &\leq& c2^{-i(1-n)} \int_{B_i} |\triangledown u| dy \\
        &\leq& c2^{-i(1-n)} ||1||_{L^{p'(\cdot)}(B_i)} ||\triangledown u||_{L^{p(\cdot)}(B_i)}, 
    \end{eqnarray*}
    where $p'(\cdot) = \frac{p(\cdot)}{p(\cdot)-1}$.
    Lemma \ref{norm of characteristic} gives $||1||_{L^{p'(\cdot)}(B_i)}\approx |B_i|^{\frac{1}{p'_{B_i}}}\leq |B_i|^{\frac{1}{{p'_{B_i}}^+}}$ and by Lemma \ref{logholder continuous} we have $|B_i|^{\frac{1}{{p'_{B_i}}^+}} \leq c |B_i|^{\frac{1}{p'(x)}} $. Therefore,
    $ ||1||_{L^{p'(\cdot)}(B_i)}\leq c(2^{-i}R)^{\frac{np(x)-n}{p(x)}}$.
    Hence, for every $i\in \mathbb{N}$, 
    \begin{equation*}
        |{ u}_{B_{i+1}} - { u}_{B_{i}}|\leq c 2^{-i(1-n+n-\frac{n}{p(x)}+\frac{n-p(x)}{p(x)})}\log^{\frac{-p^+-\delta}{p^+}}\left(\frac{1}{2^{-i}}\right)= c\log^{\frac{-p^+-\delta}{p^+}}\left(\frac{1}{2^{-i}}\right).
    \end{equation*}
    Therefore, for $k>j $ we obtain
    \begin{equation*}
       |{ u}_{B_{k}} - { u}_{B_{j}}| \leq \sum_{i=j}^{k-1}|{ u}_{B_{i+1}} - { u}_{B_{i}}|\leq c\sum_{i=j}^{k-1}\log^{\frac{-p^+-\delta}{p^+}}\left(\frac{1}{2^{-i}}\right)= c \sum_{i=j}^{k-1} i^{\frac{-p^+-\delta}{p^+}}(\log 2)^{\frac{-p^+-\delta}{p^+}}  
    \end{equation*}
    which shows that $({ u}_{B_{i}})$ is a Cauchy sequence, since the above series is finite. Hence we have 
    \begin{equation*}
        \limsup_{r\rightarrow0}\dashint_{B(x, r)} u dy < \infty
    \end{equation*}
    which together with \eqref{main equation} implies that $x\notin E$. Thus, we can write
    \begin{equation*}
        E\subset\{ x\in \mathbb{R}^{n}:\limsup_{r\rightarrow0} r^{\frac{-(n-p(x))}{p(x)}}\log^{\frac{p^++\delta}{p^+}}\left(\frac{1}{r}\right)\Vert\triangledown u\Vert_{L^{p(\cdot)}(B(x,r))}= \infty\}. 
    \end{equation*}
    For $r\leq 1$ we obtain by Lemma \ref{from logholder continuity} that 
    \begin{equation*}
    r^{\frac{-(n-p(x))}{p(x)}}\log^{\frac{p^++\delta}{p^+}}\left(\frac{1}{r}\right) \leq c r^{\frac{-(n-p(x))}{p(y)}}\log^{\frac{p^++\delta}{p^+}}\left(\frac{1}{r}\right)
    \end{equation*}
    for every $y\in B(x, r)$. Therefore 
    \begin{center}
        $r^{\frac{-(n-p(x))}{p(x)}}\log^{\frac{p^++\delta}{p^+}}\left(\frac{1}{r}\right)\Vert\triangledown u\Vert_{L^{p(\cdot)}(B(x,r))}\leq c ||r^{\frac{-(n-p(x))}{p(\cdot)}}\log^{\frac{p^++\delta}{p^+}}\left(\frac{1}{r}\right)|\triangledown u||\,|_{L^{p(\cdot)}(B(x,r))}.$
    \end{center}
    Hence 
    \begin{equation*}
        E\subset\{ x\in \mathbb{R}^{n}:\limsup_{r\rightarrow0} ||r^{\frac{-(n-p(x))}{p(\cdot)}}\log^{\frac{p^++\delta}{p^+}}\left(\frac{1}{r}\right)|\triangledown u||\,|_{L^{p(\cdot)}(B(x,r))}= \infty\}. 
    \end{equation*}
    Note that  
    \begin{equation*}
    \int_{B(x, r)}r^{\frac{-p(y)(n-p(x))}{p(y)}}\log^{\frac{p(y)(p^++\delta)}{p^+}}\left(\frac{1}{r}\right)|\triangledown u|^{p(y)} dy\leq r^{p(x)-n}\log^{p^++\delta}\left(\frac{1}{r}\right)\int_{B(x, r)} |\triangledown u|^{p(y)} dy
    \end{equation*}
    and the norm is finite if and only if the modular is, and hence we conclude 
    \begin{equation*}
        E\subset\{ x\in \mathbb{R}^{n}:\limsup_{r\rightarrow0} r^{p(x)-n}\log^{p^++\delta}\left(\frac{1}{r}\right)\int_{B(x, r)} |\triangledown u|^{p(y)} dy = \infty\}.
    \end{equation*}
     Therefore, the claim follows from Lemma \ref{main lemma}.\qed\\
\begin{remark}
If, in addition, $p(\cdot)$ also satisfies the log-log-H\"older continuity condition, then the above theorem holds for the gauge function $h(\cdot, t)=t^{n-p(\cdot)}\log^{-p(\cdot)-\delta}\left(\frac{1}{t}\right). $ The proof follows in a similar fashion, with the help of  Lemma \ref{from loglogholder continuity}. 
\end{remark}

\textbf{Proof of Theorem \ref{main theorem 1}}
For simplicity, let us take $R=2^{-m}$ for some integer $m.$ For any $\epsilon>0,$ Lemma \ref{admissible continuous} yields a continuous function $v\in W^{1,p(\cdot)}(\Omega)\cap C^{\infty}(\mathbb{R}^n)$ satisfying 
\begin{equation}\label{from the defintion of capacity}
\int_{\mathbb{R}^n} |v|^{p(y)}+ |\triangledown v|^{p(y)} dy\leq  C_{p(\cdot)}(E)+\epsilon
\end{equation}
with $v\geq 1$ on a neighbourhood of $K.$ Let us take a Lipschitz function $\phi$ on $\mathbb{R}^n$ such that $\phi=1$ on $B(x_0,2^{-m})$ and $\phi=0$ outside $B(x_0,2^{-m+1}),$ and consider another function $u=v\phi.$ Note that
  \begin{eqnarray}\label{v to u}
      \int_{\mathbb{R}^n}|\triangledown u|^{p(y)}\,dy=\int_{\mathbb{R}^n}|v\triangledown\phi+\phi\triangledown v|^{p(y)}\,dy\leq C\int_{\mathbb{R}^n}(|v|^{p(y)}+|\triangledown v|^{p(y)})\,dy
  \end{eqnarray} 
  for a positive constant $C.$ We set, for any $\delta >0,$\\
  $$E_{\delta, K_1, K_2}=\left\{x\in E: \exists\,\, \rho_x< K_2\: \text{such that} \displaystyle\int_{B(x,\rho_x)}\vert \triangledown u\vert^{p(y)}\,dy \geq K_1^{p+}\rho_x^{n-p(x)}\log^{-p(x) -\delta}(K_2/\rho_x)\right\},$$
 where we will choose $K_1\leq 1$ and $K_2$ in a way such that $E_{\delta, K_1, K_2}=E.$\\
 
   Suppose that $E\setminus E_{\delta, K_1, K_2}\neq \emptyset $ and 
   let $x\in E\setminus E_{\delta, K_1, K_2}.$ Let $\tilde{x}\in\mathbb{R}^n$ be such that $|x-\tilde{x}|=2^{-m+3}$ and take a ball $B_1$ centered at $\tilde{x}$ with radius $2^{-1}|x-\tilde{x}|=2^{-m+2}.$ Clearly, $B_1$ lies outside $B(x,2^{-m+2}).$  
   We then consider a sequence of balls (see \cite{KK15} for the existence of such balls) $B_1,\,B_2,\,B_3,\ldots$ with radii $R_1,\, R_2,\, R_3, \ldots,$ where $R_l=2^{-l}|x-\tilde{x}|$ for all $l\geq 1,$ and the centers of the balls lie on the line that joins $x$ to $\tilde{x},$ and there exists a constant $S\geq 1,$ having the following properties: 
   \begin{itemize}
       \item {$S^{-1}\diameter(B_i)\leq \dist(x, B_i)\leq S \diameter(B_i),$}
       \item {\text{There exists} $D_i\subset B_i\cap B_{i+1}$ so that $B_i\cup B_{i+1}\subset S D_i,$ }
       \item { No ball of $\mathbb{R}^n$ can belong to $S$ balls $\lambda B_i$ for every $\lambda\geq 1$.}
   \end{itemize}
   Note that the ball $B_1$ also lies outside $B(x_0,2^{-m+1}),$ since for any $y\in B_1,$
   $d(x_0,y)\geq d(x_0,y)-d(x_0,x)\geq 2^{-m+1}.$ Therefore $u_{B_1}=0,$ and hence by the telescopic argument and the classical Poincar\'e inequality,
    \begin{eqnarray*}
       1& \leq & \vert u(x)-u_{B_{1}}\vert\\
       & \leq & \sum_{l\geq 1}\vert u_{B_l}-u_{B_{l+1}}\vert\\
       & \leq & \sum_{l\geq 1}\left(\vert u_{B_l}- u_{D_l}\vert + \vert u_{B_{l+1}}- u_{D_l}\vert\right)\\
       & \leq & \sum_{l\geq 1} \left(\dashint_{D_l} \vert u- u_{B_l}\vert + \dashint_{D_l}  \vert u- u_{B_{l+1}}\vert\right)\\
       &\leq & c \sum_{l\geq 1} \dashint_{B_l}\vert u-u_{B_l}\vert dy\\
       &\leq & c \sum_{l\geq 1} \frac{\diameter(B_l)}{\vert B_l\vert} \int_{B_l} |\triangledown u| dy.
    \end{eqnarray*}   
Take $\gamma_l=:S2^{-m+1-l}+S^22^{-m+1-l},$ for $l\geq 1.$ Note that we have chosen $\gamma_l$ in such a way that $B_l\subset B(x,\gamma_l)$ for any $l\geq 1.$ Then, by H\"older's inequality,
    \begin{eqnarray}\label{before norm estimate}
 \nonumber     1 &\leq & c \sum_{l\geq 1} \frac{\diameter(B_l)}{\vert B_l\vert} \int_{B(x,\gamma_l)} |\triangledown u| dy \\ \nonumber
      &\leq& c \sum_{l\geq 1} \frac{\diameter(B_l)}{\vert B_l\vert} \Vert1\Vert_{L^{p'(\cdot)}(B(x,\gamma_l))} \Vert\triangledown u\Vert_{L^{p(\cdot)}(B(x,\gamma_l))}\\ 
       &\leq& c \sum_{l\geq 1}  (2^{-m+3-l})^{1-n}\vert B(x,\gamma_l)\vert^\frac{p(x)-1}{p(x)}\Vert\triangledown u\Vert_{L^{p(\cdot)}(B(x,\gamma_l))}.
       \end{eqnarray}
    
 Choose $K_2=2^{-m+2}S(1+S).$ Then $\gamma_l=2^{-l-1}K_2$ for all $l\geq 1.$ Since $x\in E\setminus E_{\delta, K_1, K_2},$ for all $l\geq 1 ,$ we have 
          \begin{equation}\label{gradient inequality}
            \int_{B(x,\gamma_l)}\vert \triangledown u\vert^{p(y)} dy\leq  K_1^{p^+}\gamma_l^{n-p(x)}\log^{-p^+-\delta}\left(\frac{K_2}{\gamma_l}\right)  
          \end{equation}
          and we claim that
          \begin{equation}\label{norm estimate}
          \Vert\triangledown u\Vert_{L^{p(\cdot)}(B(x,\gamma_l))} \leq K_0K_1 \gamma_l^\frac{n-p(x)}{p(x)}\log^\frac{-p^+-\delta}{p^+}(K_2/\gamma_l),
          \end{equation}
      where $$K_0:=2^{\frac{1}{p^-}}\max\{e^{c_0\frac{n-p^-}{(p^-)^2}}, \tilde{K_2}^{-\frac{n-p^-}{(p^-)^2}}\}\geq 1.$$  \\
\textbf{Proof of the claim:}          
   We divide the ball $B(x,\gamma_l)$ into two sets $B'$ and $B'',$ where $B'=\{y\in B(x,\gamma_l): p(y)\geq p(x)\}$ and $B''=B(x,\gamma_l)\setminus B'.$ In the set $B',$ we use Lemma \ref{from logholder continuity} to obtain 
   \begin{equation}\label{on B'}
   \frac{1}{\gamma_l^{p(y)}} \leq\frac{e^{c_0}}{\gamma_l^{p(x)}}\quad\text{for all}\,\, y\in B'
   \end{equation}
   while in the set $B'',$ we simply use the fact that $\gamma_l/K_2\leq 1$ to obtain 
   \begin{equation*}
   \frac{K_2^{p(y)}}{\gamma_l^{p(y)}} <\frac{K_2^{p(x)}}{\gamma_l^{p(x)}}\quad\text{for all}\,\, y\in B''
   \end{equation*}
   and hence 
  \begin{equation}\label{on B''} 
   \frac{1}{\gamma_l^{p(y)}} \leq\frac{1}{\tilde{K_2}\gamma_l^{p(x)}} \quad\text{for all}\,\, y\in B'',
   \end{equation}
   where $\tilde{K_2}=\min\{1,K_2^{p^+-p^-}\}.$
   Using \eqref{gradient inequality}, \eqref{on B'} and \eqref{on B''} we find that    
           \begin{eqnarray*}
            & & \int_{B(x,\gamma_l)}\left(\frac{\vert \triangledown u(y)\vert}{K_0K_1 \gamma_l^{\frac{n-p(x)}{p(x)}}\log^{\frac{-p^+-\delta}{p^+}}(K_2/\gamma_l)}\right)^{p(y)}dy \\
            &=& \int_{B'}\left(\frac{\vert \triangledown u(y)\vert}{K_0K_1 \gamma_l^{\frac{n-p(x)}{p(x)}}\log^{\frac{-p^+-\delta}{p^+}}(K_2/\gamma_l)}\right)^{p(y)}dy +\int_{B''} \left(\frac{\vert \triangledown u(y)\vert}{K_0K_1 \gamma_l^{\frac{n-p(x)}{p(x)}}\log^{\frac{-p^+-\delta}{p^+}}(K_2/\gamma_l)}\right)^{p(y)}dy\\
            &\leq & \int_{B'}\frac{\vert \triangledown u(y)\vert^{p(y)} e^{c_0\frac{n-p(x)}{p(x)}}}{K_0^{p^-}K_1^{p^+} \gamma_l^{n-p(x)}\log^{-p^+-\delta}(K_2/\gamma_l)}dy + \int_{B''}\frac{\vert \triangledown u(y)\vert^{p(y)}} {K_0^{p^-}K_1^{p^+}\tilde{K_2}^{\frac{n-p(x)}{p(x)}} \gamma_l^{n-p(x)}\log^{-p^+-\delta}(K_2/\gamma_l)}dy \\
            &\leq & \int_{B'}\frac{\vert \triangledown u(y)\vert^{p(y)} e^{c_0\frac{n-p^-}{p^-}}}{K_0^{p^-}K_1^{p^+} \gamma_l^{n-p(x)}\log^{-p^+-\delta}(K_2/\gamma_l)}dy + \int_{B''}\frac{\vert \triangledown u(y)\vert^{p(y)}} {K_0^{p^-}K_1^{p^+}\tilde{K_2}^{\frac{n-p^-}{p^-}} \gamma_l^{n-p(x)}\log^{-p^+-\delta}(K_2/\gamma_l)}dy \\
            &\leq & \int_{B'}\frac{\vert \triangledown u(y)\vert^{p(y)}} {2K_1^{p^+} \gamma_l^{n-p(x)}\log^{-p^+-\delta}(K_2/\gamma_l)}dy + \int_{B''}\frac{\vert \triangledown u(y)\vert^{p(y)}} {2K_1^{p^+} \gamma_l^{n-p(x)}\log^{-p^+-\delta}(K_2/\gamma_l)}dy \\
            &\leq & \frac{1}{2} + \frac{1}{2} =1,
          \end{eqnarray*}
          which establishes the norm estimate \eqref{norm estimate}.\\
     Inserting this norm estimate \eqref{norm estimate} back into \eqref{before norm estimate}, we get
     \begin{eqnarray*}
     1 &\leq &c K_0 K_1 \sum_{l\geq 1} (2^{-m+3-l})^{1-n}\left(2^{-m+1-l} S(1+S)\right)^{\frac{n(p(x)-1)}{p(x)}+\frac{n-p(x)}{p(x)}}\log^\frac{-p^+-\delta}{p^+}(2^{l+1})\\
     &\leq & c K_0 K_1 S^{n-1}(1+S)^{n-1}(\log2)^\frac{-p^+-\delta}{p^+}\sum_{l\geq 1} (l+1)^{\frac{-p^+-\delta}{p^+}}
     \end{eqnarray*}
      which gives us a contradiction, noting that the series on the right hand side converges and by choosing $ K_1<\min\bigg\{1,\frac{1}{c K_0  S^{n-1}(1+S)^{n-1}(\log2)^\frac{-p^+-\delta}{p^+}\displaystyle\sum_{l\geq 1} (l+1)^{\frac{-p^+-\delta}{p^+}}}\bigg\}$. Hence, with the above choice of $K_1$ and $K_2,$ we have $E\setminus E_{\delta, K_1, K_2}=\emptyset$ and therefore, for every $x\in E,$ there exists $\rho_x<K_2$ such that
      \begin{equation*}
            \int_{B(x,\rho_x)}\vert \triangledown u\vert^{p(y)} dy\geq K_1^{p^+}\rho_x^{n-p(x)}\log^{-p^+-\delta}\left(\frac{K_2}{\rho_x }\right).
          \end{equation*} 
           By the $5B$ covering lemma (Lemma  \ref{b covering lemma}), we can pick a collection of pairwise disjoint balls $B_i=B(x_i,\rho_{x_i})$ such that $E\subset\cup_iB(x_i,5\rho_{x_i}).$ Now summing over $i$ we get 
           \begin{equation*}
           \int_{\mathbb{R}^n}\vert \triangledown u\vert^{p(y)} dy\geq\sum_{i} \int_{B_i}\vert \triangledown u\vert^{p(y)} dy\geq K_1^{p^+}\sum_{i}\rho_{x_i}^{n-p(x_i)}\log^{-p^+-\delta}\left(\frac{K_2}{\rho_{x_i} }\right)\geq c H^{h(\cdot)}(E) 
          \end{equation*}
         and then taking $\epsilon$ to zero after combining \eqref{from the defintion of capacity} and \eqref{v to u} we get the required inequality $ H^{h(\cdot)}(E) \leq c C_{p(\cdot)}(E)$.\qed
        \vspace{1cm} 

Next, we provide two auxiliary lemmas for the proof of the last two theorems.
\begin{lemma}\label{main lemma 3}
 Let $\Phi:\mathbb{R}^n\times[0,\infty)\rightarrow[0,\infty)$  be given by $\Phi(x,t):=t^{p(x)}\log^{q(x)}(e+t)$, $q(x) >0$ for all $x\in \mathbb{R}^n$. If $||\triangledown u||_{L^{\Phi(\cdot,\cdot)}(\mathbb{R}^n)} \leq 1,$ then 
 \begin{equation}
     ||\,|\triangledown u| (\log(e+|\triangledown u|))^{\frac{q(\cdot)}{p^+}}||_{L^{p(\cdot)}(\mathbb{R}^n)} \leq 1.
 \end{equation}
\end{lemma}
\begin{proof}
    Since we have $||\triangledown u||_{L^{\Phi(\cdot,\cdot)}(\mathbb{R}^n)} \leq 1$, if we can show that $\rho_{p(\cdot)}(|\triangledown u |(\log(e+|\triangledown u|))^{\frac{q(\cdot)}{p^+}})\leq 1$ then we will get $ ||\,|\triangledown u| (\log(e+|\triangledown u|))^{\frac{q(\cdot)}{p^+}}||_{L^{p(\cdot)}(\mathbb{R}^n)} \leq 1$.\\
    Consider, 
    \begin{eqnarray*}
        \rho_{p(\cdot)}(|\triangledown u| (\log(e+|\triangledown u|))^{\frac{q(y)}{p^+}}) &=& \int_{\Omega} |\triangledown u|^{p(y)} (\log(e+|\triangledown u|))^{\frac{q(y) p(y)}{p^+}} dy\\
        &\leq& \int_{\Omega} |\triangledown u|^{p(y)} (\log^{q(y)}(e+|\triangledown u|)) dy\\
        &\leq & 1,
    \end{eqnarray*}
    since $||\triangledown u||_{L^{\Phi(\cdot,\cdot)}(\mathbb{R}^n)} \leq 1$. Therefore, we get $ ||\,|\triangledown u| (\log(e+|\triangledown u|))^{\frac{q(\cdot)}{p^+}}||_{L^{p(\cdot)}(\mathbb{R}^n)} \leq 1$.\\
\end{proof}
\begin{lemma}\label{main lemma 4}
 Let $\Phi:\mathbb{R}^n\times[0,\infty)\rightarrow[0,\infty)$  be given by $\Phi(x,t):=t^{p(x)}\log^{q(x)}(e+t)$, $q(x) <0$ for all $x\in \mathbb{R}^n$. If $||\triangledown u||_{L^{\Phi(\cdot,\cdot)}(\mathbb{R}^n)} \leq 1,$ then 
 \begin{equation}
     ||\,|\triangledown u| (\log(e+|\triangledown u|))^{\frac{q(\cdot)}{p^-}}||_{L^{p(\cdot)}(\mathbb{R}^n)} \leq 1.
 \end{equation}
\end{lemma}
\begin{proof}
    Since we have $||\triangledown u||_{L^{\Phi(\cdot,\cdot)}(\mathbb{R}^n)} \leq 1$, if we can show that $\rho_{p(\cdot)}(|\triangledown u |(\log(e+|\triangledown u|))^{\frac{q(\cdot)}{p^-}})\leq 1$ then we will get $ ||\,|\triangledown u| (\log(e+|\triangledown u|))^{\frac{q(\cdot)}{p^-}}||_{L^{p(\cdot)}(\mathbb{R}^n)} \leq 1$.\\
    Consider, 
    \begin{eqnarray*}
        \rho_{p(\cdot)}(|\triangledown u| (\log(e+|\triangledown u|))^{\frac{q(\cdot)}{p^-}}) &=& \int_{\Omega} |\triangledown u|^{p(y)} (\log(e+|\triangledown u|))^{\frac{q(y) p(y)}{p^-}} dy\\
        &\leq& \int_{\Omega} |\triangledown u|^{p(y)} (\log^{q(y)}(e+|\triangledown u|)) dy\\
        &\leq & 1,
    \end{eqnarray*}
    since $||\triangledown u||_{L^{\Phi(\cdot,\cdot)}(\mathbb{R}^n)} \leq 1$. Therefore, we get $ ||\,|\triangledown u| (\log(e+|\triangledown u|))^{\frac{q(\cdot)}{p^-}}||_{L^{p(\cdot)}(\mathbb{R}^n)} \leq 1$. \\
\end{proof}

\textbf{Proof of Theorem \ref{main theorem 3}}
 Since $C_{\Phi(\cdot, \cdot)}(E)=0,$ for every $i\in\mathbb{N},$ we have a non-negative function $u_i \in  W^{1,\Phi(\cdot,\cdot)} (\mathbb{R}^n)$ such that $E\subset \interior\{u_i\geq 1\}$ and $||u_i||_{W^{1,\Phi(\cdot,\cdot)} (\mathbb{R}^n)}\leq 2^{-i}$. Define $v_i:= u_1+....+u_i.$ Then $(v_i)$ is a Cauchy sequence and since $W^{1,\Phi(\cdot,\cdot)}(\mathbb{R}^n)$ is a Banach space, there exists $u\in W^{1,\Phi(\cdot,\cdot)}(\mathbb{R}^n)$ such that $v_i\rightarrow u$ and hence
    \begin{center}
        $||u||_{1,\Phi(\cdot,\cdot)}\leq \lim_{i\rightarrow \infty} ||v_i||_{1,\Phi(\cdot,\cdot)}\leq \sum_{i=1}^{\infty}||u_i||_{1,\Phi(\cdot,\cdot)}\leq\sum_{i=1}^{\infty}2^{-i}\leq 1. $
    \end{center}
    Therefore, by Lemma \ref{main lemma 3} we get $ ||\,|\triangledown u| (\log(e+|\triangledown u|))^{\frac{q(\cdot)}{p^+}}||_{L^{p(\cdot)}(\Omega)} \leq 1$.\\
     For every natural number $k$ and $x\in E$ we find $r$ such that $u_{i}(y) \geq 1$ for every $i= 1,....,k$ and almost every $y\in B(x, r),$ which means that  $E\subset \interior\{u\geq m\}$ for all $m\geq 1$. Then for $r$ small enough we get $B(x, r) \subset \interior\{u\geq m\}$, and hence $ \dashint_{B(x, r)} u dy \geq m$. Therefore 
    \begin{equation}\label{main equation 2}
        \limsup_{r\rightarrow0}\dashint_{B(x, r)} u dy = \infty.
    \end{equation}
    Suppose that $x$ is such that 
    \begin{equation}
        \limsup_{r\rightarrow0} r^{\frac{-(n-p(x))}{p(x)}}\log \left(\frac{1}{r}\right)\Vert |\triangledown u| (\log(e+|\triangledown u|))^{\frac{q(\cdot)}{p^+}}\Vert_{L^{p(\cdot)}(B(x,r))}=: C <\infty.
    \end{equation}
     We choose $R\in (0, 1) $ so that, for every $0<r\leq R$,
    \begin{equation*}
        \Vert|\triangledown u| (\log(e+|\triangledown u|))^{\frac{q(\cdot)}{p^+}}\Vert_{L^{p(\cdot)}(B(x,r))}< C r^{\frac{n-p(x)}{p(x)}}\log^{-1}\left(\frac{1}{r}\right).
    \end{equation*}
     Denote $B_i:= B(x, 2^{-i}R)$, $B_i^1:= \{y\in B_i: |\triangledown u(y)| \leq (2^{-i}R)^{\frac{-1}{2}}\}$ and $B_i^2:= \{y\in B_i: |\triangledown u(y)| > (2^{-i}R)^{\frac{-1}{2}}\}$. Then using  the classical Poincar\'e inequality and the H\"older inequality \eqref{holder inequality}, we have
    \begin{eqnarray*}
        & &|{ u}_{B_{i+1}} - { u}_{B_{i}}|\\
        &\leq& \dashint_{B_{i+1}}|u-{ u}_{B_{i}}| dy\\ &\leq& c2^{(i+1)n} \int_{B_i}|u-{ u}_{B_{i}}| dy\\
        &\leq& c2^{-i(1-n)} \int_{B_i} |\triangledown u| dy \\
        &=& c2^{-i(1-n)} \int_{B_i^1} |\triangledown u| dy+ c2^{-i(1-n)} \int_{B_i^2} |\triangledown u| dy\\
        &\leq& c2^{-i(1-n)}(2^{-i}R)^{\frac{-1}{2}}(2^{-i}R)^n + \frac{c2^{-i(1-n)}}{(\log(e+(2^{-i}R)^{\frac{-1}{2}}))^{\frac{q^-}{p^+}}} \int_{B_i^2} |\triangledown u| (\log(e+|\triangledown u|))^{\frac{q(y)}{p^+}} dy\\
        &\leq& c2^{\frac{-i}{2}} R^{n-\frac{1}{2}} + \frac{c2^{-i(1-n)}}{(\log(e+(2^{-i}R)^{\frac{-1}{2}}))^{\frac{q^-}{p^+}}}||1||_{L^{p'(\cdot)}(B_i)} \Vert|\triangledown u| (\log(e+|\triangledown u|))^{\frac{q(y)}{p^+}}\Vert_{L^{p(\cdot)}(B_i)}.
    \end{eqnarray*}
     Lemma \ref{norm of characteristic} gives $||1||_{L^{p'(\cdot)}(B_i)}\approx |B_i|^{\frac{1}{p'_{B_i}}}\leq |B_i|^{\frac{1}{{p'_{B_i}}^+}}$ and by Lemma \ref{logholder continuous} we have $|B_i|^{\frac{1}{{p'_{B_i}}^+}} \leq c |B_i|^{\frac{1}{p'(x)}} $. Therefore,
     $ ||1||_{L^{p'(\cdot)}(B_i)}\leq c(2^{-i}R)^{\frac{np(x)-n}{p(x)}}$.
      Hence, for every $i \in \mathbb{N}$, 
     \begin{eqnarray*}
        |{ u}_{B_{i+1}} - { u}_{B_{i}}| &\leq & c2^{\frac{-i}{2}} R^{n-\frac{1}{2}}+ c 2^{-i(1-n+n-\frac{n}{p(x)}+\frac{n-p(x)}{p(x)})}\frac{\log^{-1}\left(\frac{1}{2^{-i}}\right)}{(\log(e+(2^{-i}R)^{\frac{-1}{2}}))^{\frac{q^-}{p^+}}} \\ 
       & = & c2^{\frac{-i}{2}} R^{n-\frac{1}{2}}+c\frac{\log^{-1}\left(\frac{1}{2^{-i}}\right)}{(\log(e+(2^{-i}R)^{\frac{-1}{2}}))^{\frac{q^-}{p^+}}}.
    \end{eqnarray*}
    Therefore, for $k>j $ we obtain 
    \begin{eqnarray*}
       |{u}_{B_{k}} - { u}_{B_{j}}| &\leq & \sum_{i=j}^{k-1}|{ u}_{B_{i+1}} - { u}_{B_{i}}|\\
       &\leq & c\sum_{i=j}^{k-1}\bigg[2^{\frac{-i}{2}} R^{n-\frac{1}{2}}+\frac{\log^{-1}\left(\frac{1}{2^{-i}}\right)}{(\log(e+(2^{-i}R)^{\frac{-1}{2}}))^{\frac{q^-}{p^+}}}\bigg]\\
       &\leq &  c \sum_{i=j}^{k-1} \bigg[ 2^{\frac{-i}{2}}R^{n-\frac{1}{2}} + \frac{\log^{-1}(2)}{i^{(1+\frac{q^-}{p^+})}\log^{\frac{q^-}{p^+}}(e+2R)} \bigg]
    \end{eqnarray*}
     which shows that $({ u}_{B_{i}})$ is a Cauchy sequence, since the above series is finite. Hence we have  
    \begin{equation*}
        \limsup_{r\rightarrow0}\dashint_{B(x, r)} u dy < \infty
    \end{equation*}
    which together with \eqref{main equation 2} implies that $x\notin E$. Thus, we can write
    \begin{equation*}
        E\subset\{ x\in \mathbb{R}^{n}:\limsup_{r\rightarrow0} r^{\frac{-(n-p(x))}{p(x)}}\log \left(\frac{1}{r}\right)\Vert |\triangledown u| (\log(e+|\triangledown u|))^{\frac{q(\cdot)}{p^+}}\Vert_{L^{p(\cdot)}(B(x,r))}= \infty\}. 
    \end{equation*}
    For $r\leq 1$ we obtain by Lemma \ref{from logholder continuity} that 
    \begin{equation*}
    r^{\frac{-(n-p(x))}{p(x)}}\log \left(\frac{1}{r}\right) \leq c r^{\frac{-(n-p(x))}{p(y)}}\log \left(\frac{1}{r}\right)
    \end{equation*}
     for every $y\in B(x, r)$. Therefore
     \begin{eqnarray*}
       & & r^{\frac{-(n-p(x))}{p(x)}}\log \left(\frac{1}{r}\right)\Vert |\triangledown u| (\log(e+|\triangledown u|))^{\frac{q(\cdot)}{p^+}}\Vert_{L^{p(\cdot)}(B(x,r))}\\
       &\leq & c ||r^{\frac{-(n-p(x))}{p(\cdot)}}\log \left(\frac{1}{r}\right) |\triangledown u| (\log(e+|\triangledown u|))^{\frac{q(\cdot)}{p^+}}||_{L^{p(\cdot)}(B(x,r))}.
    \end{eqnarray*}
    Hence 
    \begin{equation*}
        E\subset\{ x\in \mathbb{R}^{n}:\limsup_{r\rightarrow0} ||r^{\frac{-(n-p(x))}{p(\cdot)}}\log \left(\frac{1}{r}\right)  |\triangledown u| (\log(e+|\triangledown u|))^{\frac{q(\cdot)}{p^+}}||_{L^{p(\cdot)}(B(x,r))}= \infty\}. 
    \end{equation*}
    Note that 
    \begin{eqnarray*}
        & &\int_{B(x, r)}r^{\frac{-p(y)(n-p(x))}{p(y)}}\log^{p(y)}\left(\frac{1}{r}\right)  (|\triangledown u| \log^{\frac{q(y)}{p^+}}(e+|\triangledown u|))^{p(y)} dy \\
     &\leq& r^{p(x)-n}\log^{p^+} \left(\frac{1}{r}\right)\int_{B(x, r)} (|\triangledown u| (\log(e+|\triangledown u|))^{\frac{q(y)}{p^+}})^{p(y)} dy\\
    \end{eqnarray*}
    and the norm is finite if and only if the modular is, and hence we conclude 
    \begin{equation*}
        E\subset\{ x\in \mathbb{R}^{n}:\limsup_{r\rightarrow0} r^{p(x)-n}\log^{p^+} \left(\frac{1}{r}\right)\int_{B(x, r)} (|\triangledown u| (\log(e+|\triangledown u|))^{\frac{q(y)}{p^+}})^{p(y)} dy = \infty\}.
    \end{equation*}
    Therefore, the claim follows from Lemma \ref{main lemma}.\qed\\

\textbf{Proof of Theorem \ref{main thorem 4}}
     Since $C_{\Phi(\cdot, \cdot)}(E)=0,$ for every $i\in\mathbb{N},$ we have a non-negative function $u_i \in  W^{1,\Phi(\cdot,\cdot)} (\mathbb{R}^n)$ such that $E\subset \interior\{u_i\geq 1\}$ and $||u_i||_{W^{1,\Phi(\cdot,\cdot)} (\mathbb{R}^n)}\leq 2^{-i}$. Define $v_i:= u_1+....+u_i.$ Then $(v_i)$ is a Cauchy sequence and since $W^{1,\Phi(\cdot,\cdot)}(\mathbb{R}^n)$ is a Banach space, there exists $u\in W^{1,\Phi(\cdot,\cdot)}(\mathbb{R}^n)$ such that $v_i\rightarrow u$ and hence
    \begin{center}
        $||u||_{1,\Phi(\cdot,\cdot)}\leq \lim_{i\rightarrow \infty} ||v_i||_{1,\Phi(\cdot,\cdot)}\leq \sum_{i=1}^{\infty}||u_i||_{1,\Phi(\cdot,\cdot)}\leq\sum_{i=1}^{\infty}2^{-i}\leq 1. $
    \end{center}
    Therefore, by Lemma \ref{main lemma 4} we get $ ||\,|\triangledown u| (\log(e+|\triangledown u|))^{\frac{q(\cdot)}{p^-}}||_{L^{p(\cdot)}(\Omega)} \leq 1$.\\
    
     For every natural number $k$ and $x\in E$ we find $r$ such that $u_{i}(y) \geq 1$ for every $i= 1,....,k$ and almost every $y\in B(x, r),$ which means that  $E\subset \interior\{u\geq m\}$ for all $m\geq 1$. Then for $r$ small enough we get $B(x, r) \subset \interior\{u\geq m\}$, and hence $ \dashint_{B(x, r)} u dy \geq m$. Therefore 
    \begin{equation}\label{main equation 3}
        \limsup_{r\rightarrow0}\dashint_{B(x, r)} u dy = \infty.
    \end{equation}
    Suppose that $x$ is such that 
    \begin{equation}
        \limsup_{r\rightarrow0} r^{\frac{-(n-p(x))}{p(x)}}\log^{1-\frac{\lambda}{p^-}+\frac{\delta}{p^+}} \left(\frac{1}{r}\right)\Vert |\triangledown u| (\log(e+|\triangledown u|))^{\frac{q(\cdot)}{p^-}}\Vert_{L^{p(\cdot)}(B(x,r))}=: C<\infty.
    \end{equation}
     We choose $R\in (0, 1) $ so that, for every $0<r\leq R$,  
    \begin{equation*}
        \Vert|\triangledown u| (\log(e+|\triangledown u|))^{\frac{q(\cdot)}{p^-}}\Vert_{L^{p(\cdot)}(B(x,r))}< C r^{\frac{n-p(x)}{p(x)}}\log^{-1+\frac{\lambda}{p^-}-\frac{\delta}{p^+}}\left(\frac{1}{r}\right).
    \end{equation*}
     Denote $B_i:= B(x, 2^{-i}R)$. Then using the classical Poincar\'e inequality, Jensen's inequality \{for convex function $\Psi(t) = t \log^{\frac{q^+}{p^-}}(e+t)$\}   and the  H\"older inequality \eqref{holder inequality}, we have  
    \begin{eqnarray*}
        |{ u}_{B_{i+1}} - { u}_{B_{i}}|
        &\leq& \dashint_{B_{i+1}}|u-{ u}_{B_{i}}| dy\\ &\leq& c2^{(i+1)n} \int_{B_i}|u-{ u}_{B_{i}}| dy\\
        &\leq& c2^{-i} \dashint_{B_i} |\triangledown u| dy \\
        &\leq & c2^{-i} \Psi^{-1} \big( \dashint_{B_i} \Psi(|\triangledown u|) dy \big)\\
        &= & c2^{-i} \frac{\dashint_{B_i} \Psi(|\triangledown u|) dy }{\log^{\frac{q^+}{p^-}}(e+ \dashint_{B_i} \Psi(|\triangledown u|) dy)}\\
        &\leq & c2^{-i(1-n)} \frac{\int_{B_i} \Psi(|\triangledown u|) dy }{\log^{\frac{q^+}{p^-}}(e+2^{-in})}\\
        &\leq & \frac{c2^{-i(1-n)}}{(\log(e+2^{-in}))^{\frac{q^+}{p^-}}}||1||_{L^{p'(\cdot)}(B_i)} \Vert|\triangledown u| (\log(e+|\triangledown u|))^{\frac{q(\cdot)}{p^-}}\Vert_{L^{p(\cdot)}(B_i)}.
        \end{eqnarray*}
        Here $\Psi^{-1}$ is the generalized inverse of $\Psi$ and we know that $\Psi^{-1}(t) \approx \frac{t}{\log^{\frac{q^+}{p^-}}(e+t)}$ .\\
     Lemma \ref{norm of characteristic} gives $||1||_{L^{p'(\cdot)}(B_i)}\approx |B_i|^{\frac{1}{p'_{B_i}}}\leq |B_i|^{\frac{1}{{p'_{B_i}}^+}}$ and by Lemma \ref{logholder continuous} we have   $|B_i|^{\frac{1}{{p'_{B_i}}^+}} \leq c |B_i|^{\frac{1}{p'(x)}} $. Therefore we get 
    $ ||1||_{L^{p'(\cdot)}(B_i)}\leq c(2^{-i}R)^{\frac{np(x)-n}{p(x)}}$.
    Hence, for every $i \in \mathbb{N}$,
     \begin{equation*}
        |{ u}_{B_{i+1}} - { u}_{B_{i}}|\leq  c 2^{-i(1-n+n-\frac{n}{p(x)}+\frac{n-p(x)}{p(x)})}\frac{\log^{-1+\frac{q^+}{p^-}-\frac{\delta}{p^+}}\left(\frac{1}{2^{-i}}\right)}{(\log(e+2^{-in}))^{\frac{q^+}{p^-}}} \\ 
        = c\frac{\log^{-1+\frac{q^+}{p^-}-\frac{\delta}{p^+}}\left(\frac{1}{2^{-i}}\right)}{(\log(e+2^{-in}))^{\frac{q^+}{p^-}}}.
    \end{equation*}
    Therefore, for $k>j $ we obtain 
    \begin{equation*}
       |{ u}_{B_{k}} - { u}_{B_{j}}| \leq \sum_{i=j}^{k-1}|{ u}_{B_{i+1}} - { u}_{B_{i}}|\leq c\sum_{i=j}^{k-1}\frac{\log^{-1+\frac{q^+}{p^-}-\frac{\delta}{p^+}}\left(\frac{1}{2^{-i}}\right)}{(\log(e+2^{-in}))^{\frac{q^+}{p^-}}}\\
       \leq  c \sum_{i=j}^{k-1} \frac{\log^{-1+\frac{q^+}{p^-}-\frac{\delta}{p^+}}(2)}{i^{(1+\frac{\delta}{p^+})}\log^{\frac{q^+}{p^-}}(e+2^{-n})}  
    \end{equation*}
     which shows that $({ u}_{B_{i}})$ is a Cauchy sequence since the above series is finite. Hence we have 
    \begin{equation*}
        \limsup_{r\rightarrow0}\dashint_{B(x, r)} u dy < \infty
    \end{equation*}
    which together with \eqref{main equation 3} implies that $x\notin E$. Thus we can write 
     \begin{equation*}
        E\subset\{ x\in \mathbb{R}^{n}:\limsup_{r\rightarrow0} r^{\frac{-(n-p(x))}{p(x)}}\log^{1-\frac{q^+}{p^-}+\frac{\delta}{p^+} }\left(\frac{1}{r}\right)\Vert |\triangledown u| (\log(e+|\triangledown u|))^{\frac{q(\cdot)}{p^-}}\Vert_{L^{p(\cdot)}(B(x,r))}= \infty\}. 
    \end{equation*}
    For $r\leq 1$ we obtain by Lemma \ref{from logholder continuity} that 
    \begin{equation*}
    r^{\frac{-(n-p(x))}{p(x)}}\log^{1-\frac{q^+}{p^-}+\frac{\delta}{p^+} } \left(\frac{1}{r}\right) \leq c r^{\frac{-(n-p(x))}{p(y)}}\log^{1-\frac{q^+}{p^-}+\frac{\delta}{p^+} } \left(\frac{1}{r}\right)
    \end{equation*}
    for every $y\in B(x, r)$. Therefore
      \begin{eqnarray*}
        & & r^{\frac{-(n-p(x))}{p(x)}}\log^{1-\frac{q^+}{p^-}+\frac{\delta}{p^+} } \left(\frac{1}{r}\right)\Vert |\triangledown u| (\log(e+|\triangledown u|))^{\frac{q(\cdot)}{p^-}}\Vert_{L^{p(\cdot)}(B(x,r))}\\
       & \leq& c ||r^{\frac{-(n-p(x))}{p(\cdot)}}\log^{1-\frac{q^+}{p^-}+\frac{\delta}{p^+} } \left(\frac{1}{r}\right) |\triangledown u| (\log(e+|\triangledown u|))^{\frac{q(\cdot)}{p^-}}||_{L^{p(\cdot)}(B(x,r))}.\\
     \end{eqnarray*}
    Hence 
    \begin{equation*}
        E\subset\{ x\in \mathbb{R}^{n}:\limsup_{r\rightarrow0} ||r^{\frac{-(n-p(x))}{p(\cdot)}}\log^{1-\frac{q^+}{p^-}+\frac{\delta}{p^+} } \left(\frac{1}{r}\right)  |\triangledown u| (\log(e+|\triangledown u|))^{\frac{q(\cdot)}{p^-}}||_{L^{p(\cdot)}(B(x,r))}= \infty\}. 
    \end{equation*}
    Note that  
    \begin{eqnarray*}
        & &\int_{B(x, r)}r^{\frac{-p(y)(n-p(x))}{p(y)}}\log^{p(y)-\frac{q^+ p(y)}{p^-}+\frac{\delta p(y)}{p^+} }\left(\frac{1}{r}\right)  (|\triangledown u| (\log(e+|\triangledown u|))^{\frac{q(y)}{p^-}})^{p(y)} dy \\
     &\leq& r^{p(x)-n}\log^{p^+-\frac{q^+ p^+}{p^-}+\delta} \left(\frac{1}{r}\right)\int_{B(x, r)} (|\triangledown u| (\log(e+|\triangledown u|))^{\frac{q(y)}{p^-}})^{p(y)} dy\\
    \end{eqnarray*}
    and the norm is finite if and only if the modular is, and hence we conclude 
    \begin{equation*}
        E\subset\{ x\in \mathbb{R}^{n}:\limsup_{r\rightarrow0} r^{p(x)-n}\log^{p^+-\frac{q^+ p^+}{p^-}+\delta} \left(\frac{1}{r}\right)\int_{B(x, r)} (|\triangledown u| (\log(e+|\triangledown u|))^{\frac{q(y)}{p^-}})^{p(y)} dy = \infty\}.
    \end{equation*}
    Therefore, the claim follows by Lemma \ref{main lemma}.\qed

  \def\bibname{References}
\bibliography{main_variable}
\bibliographystyle{alpha}

  \bigskip

\noindent{\small Ankur Pandey}\\
\small{Department of Mathematics,}\\
\small{Birla Institute of Technology and Science-Pilani, Hyderabad Campus,}\\
\small{Hyderabad-500078, India} \\
{\tt p20210424@hyderabad.bits-pilani.ac.in; pandeyankur600@gmail.com}\\

\noindent{\small Nijjwal Karak}\\
\small{Department of Mathematics,}\\
\small{Birla Institute of Technology and Science-Pilani, Hyderabad Campus,}\\
\small{Hyderabad-500078, India} \\
{\tt nijjwal@gmail.com; nijjwal@hyderabad.bits-pilani.ac.in}\\
\\
\noindent{\small Debarati Mondal}\\
\small{Department of Mathematics,}\\
\small{Birla Institute of Technology and Science-Pilani, Hyderabad Campus,}\\
\small{Hyderabad-500078, India} \\
{\tt  p20200038@hyderabad.bits-pilani.ac.in}\\
\\
\end{document}